\documentclass[]{amsart}
\usepackage{amsmath}
\usepackage[english]{babel}
\usepackage[utf8]{inputenc}
\usepackage{amsfonts}
\usepackage{amsthm}
\usepackage{color}
\numberwithin{equation}{section}

\newtheorem{thm}{Theorem}[section]

\newtheorem{prop}[thm]{Proposition}
\newtheorem{cor}[thm]{Corollary}
\newtheorem{lem}[thm]{Lemma}

\theoremstyle{remark}
\newtheorem{rmk}[thm]{Remark}
\theoremstyle{definition}
\newtheorem{defn}{Definition}[section]

\DeclareMathOperator{\E}{\mathbb{E}}
\DeclareMathOperator{\C}{\mathbb{C}}

\DeclareMathOperator{\N}{\mathbb{N}}
\DeclareMathOperator{\R}{\mathbb{R}}
\DeclareMathOperator{\cF}{\mathcal{F}}

\DeclareMathOperator{\cO}{\mathcal{O}}

\DeclareMathOperator{\cL}{\mathcal{L}}

\DeclareMathOperator{\cS}{\mathcal{S}}

\DeclareMathOperator{\bH}{\mathbb{H}}

\newcommand{\pd}[2]{\frac{\partial #1}{\partial #2}}
\newcommand{\pdsup}[3]{\frac{\partial^{#3} #1}{\partial #2^{#3}}}

\newcommand{\der}[2]{\frac{d #1}{d #2}}

\newcommand{\Norm}[2]{\left\Vert #1 \right\Vert_{#2}}

\author{Giacomo Ascione$^\ast$}
\address{$^\ast$ Dipartimento di Matematica e Applicazioni ``Renato Caccioppoli'', Universita degli Studi di Napoli Federico II, 80126 Napoli, Italy}
\author{Yuliya Mishura$^\odot $}
\address{$^\odot $ Department of Probability Theory, Statistics and Actuarial Mathematics, Taras Shevchenko National University of Kyiv, Volodymyrska 64, Kyiv 01601, Ukraine}
\author{Enrica Pirozzi$^\ast$}
\email{giacomo.ascione@unina.it \\
	myus@univ.kiev.ua \\
	enrica.pirozzi@unina.it}
\title[Fokker-Planck equations for the time-changed fOU  process]{The Fokker-Planck equation  for the time-changed fractional Ornstein-Uhlenbeck process}
\begin{document}
\maketitle
\begin{abstract}
In this paper we study some properties of the generalized Fokker-Planck equation induced by the time-changed fractional Ornstein-Uhlenbeck process. First of all, we exploit some sufficient conditions to show that a mild solution of such equation is actually a classical solution. Then we discuss an isolation result for mild solutions. Finally, we prove the weak maximum principle for strong solutions of the aforementioned equation and then a uniqueness result.
\end{abstract}
\section{Introduction}
The Ornstein-Uhlenbeck (OU) process is a standard process in the application context. However, its covariance with a fast decay and the fact that the strong Markov property holds makes it to be unrealistic in situations in which memory plays a crucial role. For this reason, in \cite{cheridito2003fractional}, the fractional OU (fOU) process has been introduced as the solution of the fractional Brownian motion (fBm)-driven equation
\begin{equation*}
dU_H(t)=-\frac{1}{\theta}U_H(t)dt+\sigma dB_H(t)
\end{equation*}
where $\theta,\sigma>0$ and $B_H(t)$ is a fBm with Hurst parameter $H \in (0,1)$. Such kind of process exhibits long or short  memory depending on the value of the Hurst parameter (see \cite{cheridito2003fractional,kaarakka2011fractional}). In the context of the applications, memory phenomena occurs for instance in the financial market, hence different kinds of noise have to be implemented to describe them (see for instance \cite{anh2005financial}). Thus, in this direction, the study of the fractional Cox-Ingersoll-Ross process, that can be expressed as the square of a fOU process until it reaches zero, has to be carried on, in particular referring to the hitting time of zero point, see \cite{mishura2018stochastic,mishura2018fractional}. On the other hand, for instance in the field of theoretical neuroscience, one can propose some different kind of noise to generate some memory effects, that are typical of neurons of the prefrontal cortex (see \cite{shinomoto1999ornstein}). In particular, correlated inputs and noises have shown to be effective to reproduce part of these memory effects (see \cite{sakai1999temporally} or \cite{ascione2019stochastic} and references therein). Moreover, other strong effects in correlation can be observed as one considers the leading input stimuli to be stochastic. For these reasons, in \cite{ascione2019fractional} we studied a fOU process with stochastic drift, considering how such drift comes into play in modifying the behaviour of the covariance.\\
From another perspective, memory has been also introduced by changing the time scale from a deterministic one to a stochastic one. This is for instance the case of \cite{leonenko2013correlation,leonenko2013fractional}, where the adjective \textit{fractional} follows from the fact that the usual Kolmogorov equations admit a fractional derivative in time if we apply a \textit{time-change} to the process, in the sense that one composes the process with the inverse of a stable subordinator. In particular in \cite{leonenko2013fractional}, the object that we will call here an $\alpha$-stable time-changed OU process has been introduced and its Kolmogorov equations have been studied. Moreover, in \cite{leonenko2013correlation}, the correlation structure is exposed, showing that   this approach leads to a long-range dependent process. A further generalization has been achieved in \cite{gajda2015time}, where a general subordinator is considered in place of a stable one. Still considering applications in neuroscience, one has to recall that another common issue linked to the role of memory is the rising of heavy-tailed distribution for spiking times (see \cite{gerstein1964random,rodieck1962some}), that are not covered by a standard Leaky Integrate-and-Fire model (see \cite{buonocore2011first}). However, under suitable assumptions on the inverse subordinator, time-changed processes exhibit this \textit{heavy-tail} property of first exit time \cite{ascione2017exit} and then time-changed OU processes revealed to be an \textit{easy-to-handle} tool to obtain such kind of spiking distribution (see \cite{ascione2019semi}).\\
In \cite{ascione2019time} we introduced a time-changed fOU process, i.e. a process obtained by considering the composition between the fOU process $U_H$ and the inverse of a subordinator, and studied some of its properties, together with its generalized Fokker-Planck equation. In particular, in \cite{ascione2019time}, we highlighted the difficulties (as also done in \cite{hahn2011time}) of \textit{time-changing} a Fokker-Planck equation with explicit dependence on time. However, we still have several open questions, in particular concerning the generalized Fokker-Planck equation we introduced.\\
First of all, we would like to express a subordination principle for solutions of the generalized Fokker-Planck equation. This means, in particular, that, given a solution $v$ of the Fokker-Planck equation associated to the fOU process with some initial-boundary data, we would like to have the following relation  for the subordinated solution: $$v_\Phi(t,x)=\int_0^{+\infty}v(s,x)f_\Phi(s,t)ds$$ for some fixed integral kernel $f_\Phi(s,t)$. In particular, we show that this holds true for \textit{mild} solutions of the Fokker-Planck equation, i.e. some weaker form of the solution defined by means of Laplace transforms. Moreover, we show that under some assumptions on the regularity of the function $v$, the subordinated function $v_\Phi$ is actually a solution of the generalized Fokker-Planck equation, thus obtaining a \textit{gain of regularity} result. Such gain of regularity result gives an improvement of \cite[Theorem $7.5$]{ascione2019time}, showing that the probability density function of the time-changed fOU process is always a classical solution of the generalized Fokker-Planck equation.\\
Another open problem concerns uniqueness of the solutions. This is a natural question, since we are actually asking if, once we have solved the generalized Fokker-Planck equation with a suitable initial datum, we actually find the probability density function of the time-changed fOU. We are not able to show uniqueness in the most weak case, i.e. the case of mild solutions, but we are still able to show an \textit{isolation result}, i.e. a non-comparability result on subordinated mild solutions. However, under stronger regularity of the involved functions, we are actually able to prove uniqueness of the solutions of initial-boundary value problems for the generalized Fokker-Planck equation, by means of a weak maximum principle.\\
Finally, let us stress that a particular attention is given to the $\alpha$-stable case, for which we are able to show the differentiability (with respect to the time variable) of the probability density function of the time-changed fOU. This differentiability result will lead to fact that such probability density function is the unique solution (under suitable initial-boundary data) of the respective generalized Fokker-Planck equation.

%
%
%
The paper is structured as follows:
\begin{itemize}
	\item In Section \ref{Sec1} we give some preliminaries concerning the property of the inverse subordinators and we recall some concepts from generalized fractional calculus;
	\item In Section \ref{Sec2} we introduce the time-changed fOU process and recall some properties that were achieved in \cite{ascione2019time};
	\item In Section \ref{Sec3} we introduce the subordination and weighted subordination operators, which will be the main tools of the paper. Moreover, we give a sufficient condition for differentiability of subordinated functions linked to the inverse $\alpha$-stable subordinator;
	\item In Section \ref{Sec4} we consider the generalized Fokker-Planck equation associated to the time-changed fOU process and we show both the subordination principle and the gain-of-regularity result;
	\item Finally, in Section \ref{Sec5} we address the problem of uniqueness of solutions, giving first an isolation result for mild solutions of the generalized Fokker-Planck equation and then a uniqueness result (and a weak maximum principle) for strong solutions.
\end{itemize}
\section{Inverse subordinators and Bernstein functions}\label{Sec1}
Let us recall the definition of subordinator, as given in \cite[Chapter $3$]{bertoin1996levy}. A subordinator $\sigma(t)$ is starting from zero, an increasing (and hence non-negative) L\'evy process. Let us denote by $\Phi(\lambda)$ its Laplace exponent, i.e. a function $\Phi:[0,+\infty) \to \R$ such that
\begin{equation*}
\E[e^{-\lambda \sigma(t)}]=e^{-t\Phi(\lambda)}, \ t\ge 0, \ \lambda \ge 0.
\end{equation*}
In particular the function $\Phi$ belongs to the class of Bernstein functions (see \cite{schilling2012bernstein}) and then   can be  represented as
\begin{equation*}
\Phi(\lambda)=a+b\lambda+\int_0^{+\infty}(1-e^{-s\lambda})\nu_\Phi(ds)
\end{equation*}
where $\nu_\Phi(dt)$ is a measure such that $\int_0^{+\infty}(t \wedge 1)\nu_\Phi(dt)<+\infty$, called the L\'evy measure of $\Phi$. Here we will consider $\Phi$ to be such that $a=b=0$ and $\nu_\Phi(0,+\infty)=+\infty$. In particular, each Bernstein function $\Phi$ determines a unique subordinator $\sigma_\Phi(y)$. Now let us define  the inverse subordinator associated to $\Phi$ as
\begin{equation*}
E_\Phi(t):=\inf\{y>0: \ \sigma_\Phi(y)>t\}.
\end{equation*}
As shown in \cite{meerschaert2008triangular}, our hypotheses on $\Phi$ are enough to guarantee that $E_\Phi(t)$ admits a probability density function $f_\Phi(s,t)$ for each $t>0$. Moreover, it has been shown that
\begin{equation}\label{laplace}
\cL_{t \to \lambda}[f_\Phi(s,t)]=\frac{\Phi(\lambda)}{\lambda}e^{-s\Phi(\lambda)}
\end{equation}
where $\cL_{t \to \lambda}$ denotes the Laplace transform operator.\\
Following the lines of \cite{kochubei2011general,toaldo2015convolution}, we can define the regularized Caputo-type non-local derivatives (induced by $\Phi$) as
\begin{equation}\label{caputo}
\partial^\Phi u(t)=\der{}{t}\int_0^t\bar{\nu}_\Phi(t-\tau)(u(\tau)-u(0))d\tau
\end{equation}
for any sufficiently regular function $u:[0,+\infty)\to \R$, where $\bar{\nu}_\Phi(t)=\nu_\Phi(t,+\infty)$. In particular, if $u$ belongs to $C^1(0,+\infty)$, then it holds
\begin{equation*}
\partial^\Phi u(t)=\int_0^t\bar{\nu}_\Phi(t-\tau)u'(\tau)d\tau.
\end{equation*}
A particular case is given by the choice $\Phi(\lambda)=\lambda^\alpha$ for $\alpha \in (0,1)$. Indeed, in this case, we get the $\alpha$-stable subordinator $\sigma_\alpha(t)$. As shown in \cite{meerschaert2013inverse}, $\sigma_\alpha(t)$ and $E_\alpha(t)$ are absolutely continuous random variables for any $t>0$ and, if we denote by $g_\alpha(s)$ the probability density function of $\sigma_\alpha(1)$ and $f_\alpha(s,t)$ the probability density function of $E_\alpha(t)$, it has been shown that
\begin{equation}\label{eqftog}
f_\alpha(s,t)=\frac{t}{\alpha}s^{-1-\frac{1}{\alpha}}g_\alpha(ts^{-\frac{1}{\alpha}}).
\end{equation}
The following lemma contains the result established in \cite{ascione2019time}, and in this paper it is proven in  Remarks $4.2$ and $4.4$, therefore now its proof is omitted.
\begin{lem}\label{lem:mominv}
Let us fix $H \in \left(\frac{1}{2},1\right)$. Then $\E[E_\alpha^{-H}(t)]<+\infty$ for any $t>0$. Moreover, for any $n>1$ it holds that $\E[E_\alpha^{-nH}(t)]=+\infty$.
\end{lem}
On the other hand, let us show the following extremal point property of the Caputo-type non-local derivative (recall that such property is already known in the case $\Phi(\lambda)=\lambda^\alpha$, as shown in \cite{luchko2009maximum}).
\begin{prop}\label{extremval}
	Let $\Phi$ be a Bernstein function regularly varying at infinity of index $\alpha \in (0,1)$. Suppose $u:[0,T]\to \R$ and there exists a maximum point $t_0$ for $u$. If $u \in W^{1,1}(0,t_0)\cap C^1((0,t_0])$ (where $W^{1,1}(0,t_0)$ is the space of absolutely continuous functions in $[0,t_0]$), then $\partial^\Phi u(t_0)\ge 0$.
\end{prop}
\begin{proof}
	Let us consider the function $g(\tau)=u(t_0)-u(\tau)$ for $\tau \in [0,T]$ and observe that $\partial^\Phi g(t_0)=-\partial^\Phi f(t_0)$. Fix $\varepsilon>0$ and write
	\begin{equation*}
	\partial^\Phi f(t_0)=\int_0^\varepsilon \bar{\nu}_\Phi(t_0-\tau)g'(\tau)d\tau+\int_\varepsilon^{t_0} \bar{\nu}_\Phi(t_0-\tau)g'(\tau)d\tau:=I_1+I_2.
	\end{equation*}
	Let us first consider $I_2$. By a change of variable we have
	\begin{equation*}
	I_2=\int_0^{t_0-\varepsilon} \bar{\nu}_\Phi(z)g'(t_0-z)dz.
	\end{equation*}
	Since we know that $g \in C^1([\varepsilon,t_0])$ for any $\varepsilon>0$, we can use dominated convergence theorem to write
	\begin{equation*}
	I_2=-\lim_{a \to 0}\int_a^{t_0-\varepsilon} \bar{\nu}_\Phi(z)dg(t_0-z).
	\end{equation*}
	Let us observe that $\bar{\nu}_\Phi$ is monotone and finite in $[a,t_0-\varepsilon]$, hence it is of bounded variation and we can use integration by parts for functions of bounded variation (see \cite{wheeden2015measure}) to obtain
	\begin{equation*}
	\int_a^{t_0-\varepsilon} \bar{\nu}_\Phi(z)dg(t_0-z)=\bar{\nu}_\Phi(t_0-\varepsilon)g(\varepsilon)-\bar{\nu}_\Phi(a)g(t_0-a)-\int_a^{t_0}g(t_0-\tau)d\nu_\Phi(\tau).
	\end{equation*}
 However, since $g(t_0)=0$ and $g \in C^1([\varepsilon,t_0])$, we know, by Lipschitz property on $[\varepsilon,t_0]$ for any $\varepsilon>0$, being the derivative of $g$ bounded in $[\varepsilon,t_0]$, that there exists a constant $K(\varepsilon)$ such that, for $a \in (0,t_0-\varepsilon)$,
	\begin{equation*}
	g(t_0-a)=g(t_0-a)-g(t_0)\le K(\varepsilon)|a|.
	\end{equation*}
	On the other hand, by Karamata's Tauberian theorem, since the Laplace transform of $\bar{\nu}_\Phi$ is actually $\Phi(\lambda)$, that is regularly varying by hypotheses, we know that $\bar{\nu}_\Phi(t)\sim \frac{\Phi(1/t)}{\Gamma(1-\alpha)}$ as $t \to 0^+$. Hence we have that $\lim_{a \to 0}\bar{\nu}_\Phi(a)g(t_0-a)=0$. Moreover, $g$ is non negative, hence, by monotone convergence theorem, we have
	\begin{equation*}
	I_2=-\bar{\nu}_\Phi(t_0-\varepsilon)g(\varepsilon)-\int_0^{t_0-\varepsilon} g(t_0-\tau)d\nu_\Phi(\tau)\le -\int_0^{t_0-\varepsilon} g(t_0-\tau)d\nu_\Phi(\tau).
	\end{equation*}
	Now fix $\varepsilon_0>0$ and suppose $\varepsilon<\varepsilon_0$. Then $t_0-\varepsilon>t_0-\varepsilon_0$ and, since the integrand is non negative,
	\begin{equation*}
	\int_0^{t_0-\varepsilon} g(t_0-\tau)d\nu_\Phi(\tau) \ge \int_0^{t_0-\varepsilon_0} g(t_0-\tau)d\nu_\Phi(\tau)=:C.
	\end{equation*}
	Thus, for any $\varepsilon \in (0,\varepsilon_0)$, it holds
	\begin{equation*}
	I_2\le -\int_0^{t_0-\varepsilon} g(t_0-\tau)d\nu_\Phi(\tau)\le -C.
	\end{equation*}
	Now let us focus on $I_1$. We have
	\begin{equation*}
	I_1\le \bar{\nu}_\Phi(t_0-\varepsilon_0)\int_0^\varepsilon |g'(\tau)|d\tau.
	\end{equation*}
	Since $g' \in L^1(0,\varepsilon)$, we know there exists $\varepsilon<\varepsilon_0$ such that $I_1 \le \frac{C}{2}$. Thus, choosing $\varepsilon>0$ as mentioned, we finally obtain
	\begin{equation*}
	\partial^\Phi g(t_0)\le -\frac{C}{2}\le 0,
	\end{equation*}
	concluding the proof.
\end{proof}
\begin{rmk}
	Let us observe that we do not actually need $\Phi$ to be regularly varying. We can prove the same result if $\Phi$ is a complete Bernstein function (hence $\nu_\Phi$ is absolutely continuous) and there exist $\alpha \in (0,1)$, $C\ge 1$ and $\ell:[1,+\infty)\to [0,+\infty)$ a slowly varying functions at infinity such that for any $\lambda \ge 1$ it holds
	\begin{equation*}
	\frac{1}{C}\le \frac{\Phi(\lambda)}{\lambda^\alpha \ell(\lambda)}\le C.
	\end{equation*}
	Indeed, by \cite[Theorem $2.10$]{kim2012potential}, we know that this is enough to achieve the needed bound on $\bar{\nu}_\Phi$.\\
	In general, for the previous proposition to hold, it is sufficient that $\lim_{t \to 0^+}t\bar{\nu}_{\Phi}(t)=0$.
\end{rmk}
Finally, let us recall that if $u:[0,+\infty) \to \R$ and $\partial^\Phi u$ are Laplace transformable, then
\begin{equation*}
	\cL[\partial^\Phi u](\lambda)=\Phi(\lambda)\cL[u](\lambda)-\frac{\Phi(\lambda)}{\lambda}u(0^+)
\end{equation*}
for any $\lambda$ such that the real part of $\lambda$ is greater or equal than the maximum of the abscissae of convergence of the Laplace transforms of $u$ and $\partial^\Phi u$.
\section{The time-changed fractional Ornstein-Uhlenbeck process}\label{Sec2}
Let $(\Omega, \mathcal{F}, P)$ be a complete probability space supporting all the stochastic processes that will be considered below. Let us fix Hurst index $H \in \left(\frac{1}{2},1\right)$  and consider a fractional Brownian motion $B^H=\{B^H(t), t\ge 0\}$ with Hurst index $H$, i.e. a centered Gaussian process with covariance function given by $$\E[B^H(t)B^H(s)]=1/2(t^{2H}+s^{2H}-|t-s|^{2H}), s,t \in \mathbb{R}^+.$$
Let us also fix some number $\theta>0$ and introduce the fractional Ornstein-Uhlenbeck process (defined in \cite{cheridito2003fractional}) as
\begin{equation*}
U_H(t)=e^{-\frac{t}{\theta}}\int_0^{t}e^{\frac{s}{\theta}}dB^H(s), t\ge 0.
\end{equation*}
Now, as done in \cite{ascione2019time} we can define the time-changed fractional Ornstein-Uhlenbeck process by considering a fractional Ornstein-Uhlenbeck process $U_H(t)$, together with  an independent inverse subordinator $E_\Phi(t)$, and   defining
\begin{equation*}
U_{H,\Phi}(t):=U_H(E_\Phi(t)).
\end{equation*}
Since it will be useful in what follows, let us recall the expression of the variance of $U_{H}(t)$:
$$V_{2,H}(t)=e^{-2\frac{t}{\theta}}\int_0^{t}\int_0^{t}e^{\frac{v+u}{\theta}}|u-v|^{2H-2}dudv, t\ge 0.$$ Moreover, let us state some properties of $V_{2n,H,\Phi}(t):=\E[|U_{H,\Phi}(t)|^{2n}]$ (see \cite[Lemma $3.1$]{ascione2019time}).
\begin{prop}
	\begin{itemize}
		\item [$(i)$] $V_{2n,H,\Phi}(t)$ is finite for any $t>0$ and $n \in \N$.
		\item [$(ii)$] It holds that
		\begin{equation*}
		V_{2n,H,\Phi}(t)=\int_0^{+\infty}V_{2n,H}(s)f_\alpha(s,t)ds.
		\end{equation*}
			\item [$(iii)$] $V_{2n,H,\Phi}(t)$ is increasing in $t$ for any $n \in \N$ and
		\begin{equation*}
		\lim_{t \to +\infty}V_{2n,H,\Phi}(t)=V_{2n,H}(\infty)=\frac{\left(2\theta^{2H}H\Gamma(2H)\right)^n\Gamma\left(\frac{2n+1}{2}\right)}{\sqrt{\pi}}.
		\end{equation*}
	\end{itemize}
\end{prop}
\begin{rmk} The fact that in property $(iii)$ the asymptotic value does not depend on $\Phi$ is strictly connected to the nature of the time-change. Indeed, $E_\Phi(t)$ acts as a delay in the time-scale of $U_H(t)$, hence we expect $U_{H,\Phi}(t)$ to have the same asymptotic behaviour, despite behaving quite differently on the whole trajectories.
\end{rmk}
We will also need the following limits  for $V_{2,H}(t)$ and its derivative. They can  be obtained by \cite[Equation $29$]{ascione2019fractional} and \cite[Lemma $5.2$]{ascione2019time}.
\begin{lem}\label{lem:Varlem}
	Function $V_{2,H}$ satisfies the relations $$\lim_{t \to 0^+}\frac{V_{2,H}(t)}{t^{2H}}=1\; \text{and} \; \lim_{t \to +\infty}V_{2,H}(t)=\theta^{2H}H\Gamma(2H).$$ Moreover, $V_{2,H} \in C^1(0,+\infty)$ and its derivative satisfies the relations
	\begin{align*}
	\lim_{t \to 0^+}\frac{V'_{2,H}(t)}{t^{2H-1}}=2H\;  \text{and}\;\lim_{t \to +\infty}e^{\frac{t}{\theta}}t^{2-2H}V'_{2,H}(t)=2H(2H-1)\theta.
	\end{align*}
\end{lem}
In \cite{ascione2019time} we introduced some other operators linked to the fractional Ornstein-Uhlenbeck process and its time-change. First of all, we needed to introduce a weighted Laplace transform
\begin{equation*}
L_Hu(\lambda)=\cL_{t \to \lambda}[V_{2,H}'(t)u(t)](\lambda), \ \lambda \in \bH,
\end{equation*}
where $\bH=\{\lambda \in \C: \ \Re(\lambda)>0\}$. On the other hand, by using a \textit{mixed} inversion technique, we expressed $L_H$ in an alternative form and then defined the operator $\widehat{L}_H$ as
\begin{multline}\label{defhatL}
\widehat{L}_{H,\Phi}v(x,\lambda)=\frac{1}{4\pi^2}\int_0^{+\infty}e^{-\lambda^\alpha t}\\\times \lim_{R \to +\infty}\int_{-\infty}^{+\infty}e^{(c_1+iw)t}\int_{-R}^{R}\cL_{t \to \lambda}[V'_{2,H}(t)](c_1-c_2+i(w-z))\\ \times\frac{\Phi^{-1}(c_2+iz)}{c_2+iz}v(\Phi^{-1}(c_2+iz))dzdwdt,
\end{multline}
where $c_1<0<c_2$, $c_1-c_2>-\frac{1}{\theta}$, the inverse of $\Phi$ is well-defined on the vertical line $r_{c_2}=\{\lambda \in \C: \ \lambda=c_2+iz, \ z \in \R\}$ and the function $\frac{\Phi^{-1}(c_2+iz)}{c_2+iz}v(x,\Phi^{-1}(c_2+iz))$ does not depend on the choice of the local inverse of $\Phi$ on $r_{c_2}$. Let us remark that the fact that $\Phi$ is invertible on a vertical line of the form $r_{c_2}$ can be ensured by the fact that any Bernstein function admits an holomorphic extension $\Phi:\bH \to \bH$, hence there exists at least a vertical line $r_{c_2}$ on which the derivative $\Phi'$ of $\Phi$ does never touch $0$.  The operator $\widehat{L}_{H,\Phi}$ takes into account, in a certain sense, the action of the time-change in $L_H$.  Indeed, it has been shown in \cite{ascione2019time} that
\begin{equation*}
\widehat{L}_{H,\Phi}\bar{p}_{H,\Phi}(\lambda,x)=L_Hp_H(\Phi(\lambda),x),
\end{equation*}
where $p_H(x,t)$ is the probability density function of $U_H(t)$ and $p_{H,\Phi}(x,t)$ is the probability density function of $U_{H,\Phi}(t)$ (that exists by \cite[Proposition $4.1$]{ascione2019time}). Finally, we can define the following operator
\begin{equation}\label{Fokker}
\cF_{H,\Phi}v(x,t)=\cL_{\lambda \to t}^{-1}\left[\frac{\Phi(\lambda)}{\lambda}\pdsup{}{x}{2}\widehat{L}_{H,\Phi}(\cL_{t \to \lambda}[v(x,t)])\right](t),
\end{equation}
which plays the role of the Fokker-Planck operator in this case. Indeed, denoting by $D(\cF_{H,\Phi})$ the domain of the aforementioned operator, in \cite{ascione2019time} it has been shown that $p_{H,\Phi}(x,t)$ belongs to $D(\cF_{H,\Phi})$ and solves
\begin{equation}\label{genFPpH}
\partial_t^\Phi p_{H,\Phi}(x,t)=\frac{1}{2}\cF_{H,\Phi}p_{H,\Phi}(x,t), \ t>0, \ x \in \R \setminus \{0\}.
\end{equation}
Here we want to study some further properties of the generalized Fokker-Planck equation \eqref{genFPpH}, focusing on uniqueness of solutions and gain of regularity.
\section{Subordination and weighted subordination}\label{Sec3}
Let us fix a Bernstein function $\Phi$ and let us introduce two operators that will be useful.
\begin{defn}
	Let $X$ be a real Banach space and $L^\infty(\R^+;X)$ the Banach space of measurable functions $v:[0,+\infty)\to X$ such that $$\Norm{v}{L^\infty(\R^+;X)}=\sup_{t \in (0,+\infty)}\Norm{v(t)}{X}<+\infty,$$
	where with $\sup$ we intend the essential supremum. Then we can define the \textbf{$\Phi$-subordination operator} $S_\Phi: L^\infty(\R^+;X) \mapsto L^\infty(\R^+;X)$ as
	\begin{equation*}
	S_\Phi v(t)=\E[v(E_\Phi(t))]=\int_0^{+\infty}v(s)f_\Phi(s,t)ds,
	\end{equation*}
	where the integral has to be interpreted in Bochner's sense.\\
	Moreover, let us define the \textbf{weighted $\Phi$-subordination operator} as
	\begin{equation*}
	S_{\Phi, H} v(t)=\E[V'_{2,H}(E_\Phi(t))v(E_\Phi(t))]=\int_0^{+\infty}V'_{2,H}(s)v(s)f_\Phi(s,t)ds.
	\end{equation*}
\end{defn}
First, let us establish   some basic properties of $S_\Phi$ and $S_{\Phi, H}$.
\begin{lem}
	The operators $S_{\Phi}$ and $S_{\Phi, H}$ are continuous from $L^\infty(\R^; X)$ to itself. In particular, $$\Norm{S_{\Phi}}{L(L^\infty(\R^+;X);L^\infty(\R^+;X))}\le 1$$ and
	$$\Norm{S_{\Phi, H}}{L(L^\infty(\R^+;X);L^\infty(\R^+;X))}\le \Norm{V_{2,H}'}{L^\infty(\R^+)}.$$
\end{lem}
\begin{proof}
Let us observe that $S_{\Phi,H}v=S_\Phi(V'_{2,H}v)$. Thus we only have to prove the property for $S_\Phi$, since $S_{\Phi,H}$ is the composition of $S_\Phi$ with a multiplication operator. In particular, by Bochner's Theorem (see \cite{arendt}), we have
\begin{equation*}
\Norm{S_\Phi v(t)}{X}\le \int_0^{+\infty}\Norm{v(s)}{X}f_\Phi(s,t)ds\le \Norm{v}{L^\infty(\R^+;X)}.
\end{equation*}
Taking the supremum we conclude the proof.
\end{proof}
Now let us show how Laplace transform acts on subordinated functions. In particular, we can prove the following result.
\begin{prop}\label{Proplaptrans}
	Let $v \in L^\infty(\R^+;X)$. Then, for any $\lambda \in \bH$, it holds
	\begin{equation*}
	\cL[S_\Phi v](\lambda)=\frac{\Phi(\lambda)}{\lambda}\cL[v](\Phi(\lambda)).
	\end{equation*}
	Moreover both $S_\Phi$ and $S_{\Phi,H}$ are injective.
\end{prop}
\begin{proof}
	Let us suppose, without loss of generality, that $\lambda \in \R^+$. Then we have
	\begin{equation*}
	\cL[S_\Phi v](\lambda)=\int_0^{+\infty}e^{-t\lambda}\int_0^{+\infty}v(s)f_\Phi(s,t)dsdt.
	\end{equation*}
	Now we want to use Fubini's theorem. To do this, let us observe that
	\begin{equation*}
	\int_0^{+\infty}\int_0^{+\infty}e^{-t\lambda}\Norm{v(s)}{X}f_\Phi(s,t)dtds\le \frac{\Norm{v}{L^\infty(\R^+;X)}}{\lambda}.
	\end{equation*}
	Thus, by using Fubini's theorem and equation \eqref{laplace}, we obtain
	\begin{align*}
	\cL[S_\Phi v](\lambda)&=\int_0^{+\infty}e^{-t\lambda}\int_0^{+\infty}v(s)f_\Phi(s,t)dsdt\\
	&=\int_0^{+\infty}v(s)\int_0^{+\infty}e^{-t\lambda}f_\Phi(s,t)dtds\\
	 &=\frac{\Phi(\lambda)}{\lambda}\int_0^{+\infty}e^{-s\Phi(\lambda)}v(s)ds=\frac{\Phi(\lambda)}{\lambda}\cL[v](\Phi(\lambda)),
	\end{align*}
	that is well defined since $\Phi:\bH \to \bH$.\\
	Now let us show the injectivity of $S_\Phi$. Since $S_\Phi$ is a linear operator, we have only to show that ${\rm Ker}S_\Phi=\{0\}$. Thus, let us suppose that $S_\Phi v=0$. Then we have
	\begin{equation*}
	\frac{\Phi(\lambda)}{\lambda}\cL[v](\Phi(\lambda))=0
	\end{equation*}
	that implies, since $\Phi(\lambda)>0$ for any $\lambda>0$, $\cL[v](\Phi(\lambda))=0$ for any $\lambda>0$. However, we have that $\Phi:\R^+ \to \R^+$ is $C^\infty(0,+\infty)$, invertible and such that $\lim_{\lambda \to 0^+}\Phi(\lambda)=0$ and $\lim_{\lambda \to +\infty}\Phi(\lambda)=+\infty$. Thus, if we choose $\lambda=\Phi^{-1}(\eta)$, we have that $\cL[v](\eta)=0$ for any $\eta>0$. By injectivity of the Laplace transform we conclude that $v \equiv 0$.\\
	Finally, let us show that $S_{\Phi,H}$ is injective. Let us suppose, as before, that $S_{\Phi,H}v\equiv 0$. Then, since $S_{\Phi,H}v=S_{\Phi}(V'_{2,H}v)$, we have that $V'_{2,H}v \equiv 0$. However, since $V'_{2,H}(t)>0$ for any $t \in (0,+\infty)$, this implies that $v \equiv 0$, concluding the proof.
\end{proof}
We can also easily exploit the link between the weighted Laplace transform and the weighted subordination operator, by means of the following Corollary.
\begin{cor}
	Let $v \in L^\infty(\R^+;X)$. Then, for any $\lambda \in \bH$, it holds
	\begin{equation*}
	\cL[S_{\Phi,H} v](\lambda)=\frac{\Phi(\lambda)}{\lambda}L_H v(\Phi(\lambda)).
	\end{equation*}
\end{cor}
\begin{proof}
	It easily follows from Proposition \ref{Proplaptrans}, the fact that $S_{\Phi,H}v=S_\Phi(V'_{2,H}v)$ and the definition of $L_H$.
\end{proof}
Let us stress out that we have proved in \cite[Proposituion $4.1$]{ascione2019time} that $p_{H,\Phi}=S_\Phi p_H$, where we can consider $X=C^2(I)$ for any compact interval $I \subseteq \R \setminus \{0\}$.
\subsection{Subordinated functions and the operator $\widehat{L}_{\Phi,H}$}
In this subsection we want to stress the link between subordinated functions and the operator $\widehat{L}_{\Phi,H}$. To do that, we first need the following Proposition.
\begin{prop}\label{altformLh}
	Fix $c_1<0<c_2$ with $c_1-c_2>-1/\theta$. Consider $v \in L^\infty(\R^+;X)$ and suppose one of the following properties holds:
	\begin{itemize}
		\item[$(a)$] $v$ is Lipschitz and $x \in \R \mapsto \cL[v](c_2+ix)$ belongs to $L^1(\R)$;
		\item[$(b)$] $v$ belongs to $L^2(\R^+)$ and $x \in \R \mapsto \cL[v](c_2+ix)$ belongs to $L^2(\R)$.
	\end{itemize}
Then it holds
\begin{multline}\label{altL}
L_Hv(x,\lambda)=\frac{1}{4\pi^2}\int_0^{+\infty}e^{-\lambda^\alpha t}\\\times \lim_{R \to +\infty}\int_{-\infty}^{+\infty}e^{(c_1+iw)t}\int_{-R}^{R}\cL_{t \to \lambda}[V'_{2,H}(t)](c_1-c_2+i(w-z))\\ \times \cL[v](c_2+iz)dzdwdt.
\end{multline}
\end{prop}
We omit the proof since it is the same of \cite[Proposition $6.6$]{ascione2019time}.
%
%
By using this alternative representation of the operator $L_H$, we can stress out how it works on subordinated functions.
\begin{prop}\label{prop:hatnonhat}
	Let $v_\Phi=S_\Phi v$ where $v \in L^\infty(\R^+;X)$ satisfies one of hypotheses $(a)$ or $(b)$ of the previous Proposition. Then for any $\lambda \in \bH$ it holds
	\begin{equation*}
	\widehat{L}_{\Phi,H}(\cL[v_\Phi])(\lambda)=L_H v(\Phi(\lambda)).
	\end{equation*}
\end{prop}
\begin{proof}
	Let us observe that, by Proposition \ref{Proplaptrans}, it holds
	\begin{equation*}
	\cL[v_\Phi](\lambda)=\frac{\Phi(\lambda)}{\lambda}\cL[v](\Phi(\lambda))
	\end{equation*}
	for any $\lambda \in \bH$. In particular it holds for $\lambda \in \Phi^{-1}(r_{c_2})$ where $\Phi$ is invertible on $r_{c_2}$, and in particular it holds $\Phi(\lambda)=c_2+iz$ for some $z \in \R$. Thus we have
	\begin{equation*}
	\frac{\Phi^{-1}(c_2+iz)}{c_2+iz}\cL[v_\Phi](\Phi^{-1}(c_2+iz))=\cL[v](c_2+iz).
	\end{equation*}
	Applying this relation to Equation \eqref{defhatL} we obtain \eqref{altL}, concluding the proof.
\end{proof}
Let us recall that we have shown in \cite{ascione2019time} that $p_H(t,x)$ is Lipschitz in $t$, thus we are under the hypotheses of the two previous Propositions.
\subsection{Derivative of $\alpha$-subordinated function}
In general we have that $S_\Phi v(t)$ could be derivable or not regardless of the fact that the function $v$ belongs to $C^1(0,+\infty)$. Indeed, the differentiability of $S_\Phi v(t)$ with respect to $t>0$ directly depends on the differentiability of the density $f_\Phi(s,t)$ of the inverse subordinator. However, in the case of the $\alpha$-stable subordinator, we can exploit a sufficient condition for differentiability of $S_\alpha v(t)$.
\begin{prop}\label{derprop}
Let $v \in C^1(0,+\infty)$. If there exist two constants $C>0$ and $\beta \in \left(\frac{\alpha-1}{\alpha},2\right)$ such that $|v'(t)|\le C t^{-\beta}$ for $t>0$, then $S_\alpha v$ is in $C^1(0,+\infty)$ and its derivative equals
\begin{equation*}
\der{}{t}S_\alpha v(t)=\alpha t^{-1}S_\alpha(zv'(z))(t), \ t>0.
\end{equation*}
\end{prop}
\begin{proof}
	Applying equation \eqref{eqftog} and    changing  variables $ts^{-1/\alpha}=w$, we arrive to equality
	\begin{equation*}
	S_\alpha v(t)=\int_0^{+\infty}v\left(\left(\frac{t}{w}\right)^\alpha\right)g_\alpha(w)dw.
	\end{equation*}
	Let's check  that we can differentiate under the  sign of integral. In order to do this, let us observe that
	\begin{equation*}
	\der{}{t}v\left(\left(\frac{t}{w}\right)^\alpha\right)=\alpha t^{\alpha-1}w^{-\alpha}v'\left(\left(\frac{t}{w}\right)^\alpha\right).
	\end{equation*}
	In particular we obtain
	\begin{equation*}
	\left|\der{}{t}v\left(\left(\frac{t}{w}\right)^\alpha\right)\right|\le C\alpha t^{\alpha(1-\beta)-1}w^{\alpha(\beta-1)}.
	\end{equation*}
	Let us consider $0<t_1<t_2$ and $t \in [t_1,t_2]$. Set
	\begin{equation*}
	t_*=\begin{cases}
	t_1 & \alpha(1-\beta)-1<0\\
	t_2 & \alpha(1-\beta)-1\ge 0
	\end{cases}
	\end{equation*}
	and $C_*=\alpha C t_*^{\alpha(1-\beta)-1}$ to obtain
	\begin{equation*}
	\left|\der{}{t}v\left(\left(\frac{t}{w}\right)^\alpha\right)\right|\le C_*w^{\alpha(\beta-1)}.
	\end{equation*}
	Finally, by definition of $\beta \in (\alpha-1,2\alpha)$, we have that $\E[\sigma_\alpha(1)^{\alpha(\beta-1)}]<+\infty$ and then we can differentiate under the integral sign, concluding the proof.
\end{proof}
\begin{rmk}
	The previous proposition holds true even if there exists a point $T>0$ such that $v \in C^1(0,T) \cap C^1(T,+\infty)$ and $\lim_{t \to T^\pm}v'(t) \in \R$ (still under the hypothesis $|v'(t)|\le C t^{-\beta}$).
\end{rmk}
It was established  in \cite{ascione2019time} that for fixed $x \not = 0$  function $t \mapsto p_H(x,t)$ is Lipschitz. Thus we have that for fixed $x \not = 0$ the function $p_{H,\alpha}(x,t)=S_{\alpha}p_H(x,t)$ is $C^1(0,+\infty)$ in $t$.
\section{The generalized Fokker-Planck equation}\label{Sec4}
We have introduced all the main tools and  can consider  the generalized Fokker-Planck equation associated to the time-changed fractional Ornstein-Uhlenbeck process. First of all, let us give   definitions of different kind of solutions for such equation. Here with \textit{classical solution} we intend the equivalent of a strong solution of a partial differential equation (see, for instance, \cite{gilbarg2015elliptic}), while with \textit{strong solution} we intend a more regular classical solution. Moreover, we refer to \cite{meerschaert2011distributed,gajda2015time} for the definition of \textit{mild solution}. The form of this equation is dictated to us by the type of operators \eqref{caputo} and \eqref{Fokker}.
\begin{defn}
	We say that $v:I \times [0,+\infty) \to \R$ (where $I \subseteq \R$ is an interval) is a \textbf{classical solution} of
	\begin{equation}\label{genFP}
	\partial_t^\Phi v(x,t)=\frac{1}{2}\cF_{\Phi,H}v(x,t), \ (x,t)\in I \times (0,+\infty)
	\end{equation}
	if
	\begin{itemize}
		\item $v$ belongs to the domain of $\cF_{\Phi,H}$;
		\item $\partial^\Phi_t v(x,\cdot)$ is well-defined for any $x \in I$;
		\item Equation \eqref{genFP} holds pointwise for almost any $t \in [0,T]$ and any $x \in I$.
	\end{itemize}
Moreover, we say that a classical solution $v$ is a \textbf{strong solution} if, for any $x \in I$, $v(x,\cdot)\in C^1(0,+\infty)$ and there exists $\varepsilon>0$ such that $v(x,\cdot)\in W^{1,1}(0,\varepsilon)$.\\
We say that $v$ is a \textbf{mild solution} of Equation \eqref{genFP} if
\begin{itemize}
	\item $v(x,\cdot)$ is Laplace transformable for any $x \in I$;
	\item The Laplace transform $\bar{v}$ of $v$ belongs to the domain of $D(\widetilde{L}_{\Phi,H})$ for any $x \in I$;
	\item For any $\lambda \in \bH$ it holds $\bar{v}(\cdot,\lambda) \in C(I)$, where $C(I)$ is the space of continuous functions in $I$;
	\item It holds
	\begin{equation}\label{ltgenFP}
		 \Phi(\lambda)\bar{v}(x,\lambda)-\frac{\Phi(\lambda)}{\lambda}v(x,0)=\frac{\Phi(\lambda)}{2\lambda}\widehat{L}_{\Phi,H}\bar{v}(x,\lambda), \ x \in I, \ \lambda \in \bH.
	\end{equation}
\end{itemize}
\end{defn}
\begin{rmk}
	Let us observe that the definition of classical solution recalls the one of Caratheodory solution for an ordinary differential equation (see \cite{coddington1955theory}), while the definition of strong solution coincides with the usual one for partial differential equations. Concerning mild solutions, let us observe that Equation \eqref{ltgenFP} arises as we take the Laplace transform on both sides of equation \eqref{genFP}. Thus we have the following chain of implications:
	\begin{equation*}
	\mbox{strong solution }\Rightarrow \mbox{ classical solution } \Rightarrow \mbox{ mild solution.}
	\end{equation*}
\end{rmk}
Our first aim is to exploit some sufficient conditions to revert some of these implications. From now on we will only consider solutions $v_\Phi$ such that there exists a Lipschitz function $v$ for which $v_\Phi=S_\Phi v$. First of all, let us stress that, in the $\alpha$-stable case, we easily have
\begin{equation*}
\mbox{strong solution }\Leftarrow \mbox{ classical solution }+ \mbox{ Proposition \ref{derprop}}.
\end{equation*}
In the next subsection we will consider some \textit{gain of regularity} results to pass from mild solutions to classical ones.
\subsection{Gain of regularity for subordinated mild solutions}
To study the gain of regularity for subordinated mild solutions, we first need to introduce the notion of mild solution for the usual Fokker-Planck equation.
\begin{defn}
	We say that $v:I \times [0,+\infty) \to \R$ is a \textbf{classical solution} of
	\begin{equation}\label{nongenFP}
	\partial_t v(x,t)=\frac{1}{2}V'_{2,H}(t)\pdsup{}{x}{2}v(x,t), \ (x,t)\in I \times (0,+\infty)
	\end{equation}
	if
	\begin{itemize}
		\item $v(t,\cdot)$ belongs to $C^2(I)$;
		\item $\partial_t v(x,\cdot)$ belongs to $L^1_{loc}(0,+\infty)$;
		\item Equation \eqref{nongenFP} holds pointwise for almost all $t>0$ and any $x \in I$.
	\end{itemize}
	Moreover, we say that a classical solution $v$ is a strong solution if $v(x,\cdot)\in C^1(0,+\infty)$.\\
	We say that $v$ is a \textbf{mild solution} of Equation \eqref{nongenFP} if
	\begin{itemize}
		\item $v(x,\cdot)$ is Laplace transformable for any $x \in I$ with Laplace transform $\bar{v}$;
		\item For any $\lambda \in \bH$ it holds $\bar{v}(\cdot,\lambda) \in C(I)$;
		\item It holds
		\begin{equation}\label{ltnongenFP}
		\lambda\bar{v}(x,\lambda)-v(x,0)=\frac{1}{2}L_{H}v(x,\lambda), \ x \in I, \ \lambda \in \bH.
		\end{equation}
	\end{itemize}
\end{defn}
Now let us show how the notion of subordinated mild solution for Equation \eqref{genFP} is linked to the one for Equation \eqref{nongenFP}.
\begin{prop}\label{propalphato1}
	Let $v_\Phi=S_\Phi v$ with $v$ satisfying the hypotheses of Proposition \ref{altformLh}. Then the following properties are equivalent:
	\begin{enumerate}
		\item $v_\Phi$ is a mild solution of \eqref{genFP};
		\item $v$ is a mild solution of \eqref{nongenFP}.
	\end{enumerate}
\end{prop}
\begin{proof}
	Let us first observe that by hypotheses $v$ is Laplace transformable, so also $v_\Phi$ is Laplace transformable. Moreover, since $v$ satisfies the hypotheses of Proposition \ref{altformLh}, by Proposition \ref{prop:hatnonhat} we have that $\bar{v}_\Phi$ belongs to the domain of $\widehat{L}_{\Phi,H}$. Moreover, by Proposition \ref{Proplaptrans} we know that $\bar{v}_\Phi$ is continuous in $x$ if and only if $\bar{v}$ is continuous in $x$. Thus we only have to show the equivalence of Equations \eqref{ltgenFP} and \eqref{ltnongenFP}. Let us only show that $(1)$ implies $(2)$, since the converse is analogous. To do this, just observe that, since $v_\Phi$ is a subordinated mild solution of \eqref{genFP}, equation \eqref{ltgenFP} holds, that is to say, by also using Proposition \ref{Proplaptrans} and the fact that, by definition $v_\Phi(x,0)=v(x,0)$, for any $x \in I$ and $\lambda \in \bH$
	\begin{equation*}
	 \frac{\Phi(\lambda)}{\lambda}(\Phi(\lambda)\overline{v}(x,\Phi(\lambda))-v(x,0))=\frac{\Phi(\lambda)}{2\lambda}\pdsup{}{x}{2}L_Hv(x,\Phi(\lambda)),
	\end{equation*}
	where $\overline{v}(x,\lambda)$ is the Laplace transform of $v(x,t)$.\\
	Without loss of generality, we can suppose $\lambda>0$ is real. Then we can multiply last relation by $\frac{\lambda}{\Phi(\lambda)}$ and write $\lambda$ in place of $\Phi(\lambda)$ (since $\Phi:[0,+\infty) \to [0,+\infty)$ is invertible), obtaining
	\begin{equation*}
	\lambda\overline{v}(x,\lambda)-v(x,0)=\frac{1}{2}\pdsup{}{x}{2}L_Hv(x,\lambda)
	\end{equation*}
	that is equation \eqref{ltnongenFP} for $v$.
\end{proof}
To show the gain of regularity result, we first need to express the Fokker-Planck operator $\cF_{\Phi,H}$ in terms of the weighted subordination operator $S_{\Phi,H}$.
\begin{lem}\label{lem:weitoFok}
	Let $v \in L^\infty(\R^+;C(I))$ and $v_\Phi=S_\Phi v$, with $v$ satisfying the hypotheses of Proposition \ref{altformLh}, be a mild solution of \eqref{genFP} such that $v_\Phi \in D(\cF_{\Phi,H},I)$ and $\cF_{\Phi,H}v_\Phi(\cdot,t)\in C(I)$ for any fixed $t>0$. Then it holds $S_{\Phi,H}v(\cdot,t) \in C^2(I)$ and, for any $x \in I$ and almost any $t \in I$,
	\begin{equation*}
	\cF_{\Phi,H}v_\Phi(x,t)=\pdsup{}{x}{2}S_{\Phi,H}v(x,t).
	\end{equation*}
\end{lem}
\begin{proof}
	Let us consider $\lambda>0$ without loss of generality. Since $v_\Phi$ is a mild solution of \eqref{genFP}, by using Proposition \ref{prop:hatnonhat}, we obtain that $\frac{\Phi(\lambda)}{\lambda}L_Hv(\cdot ,\Phi(\lambda)) \in C^2(I)$. Since $v_\Phi$ belongs to the domain of $\cF_{\Phi,H}$, we have that $\pdsup{}{x}{2}\frac{\Phi(\lambda)}{\lambda}L_Hv(x ,\Phi(\lambda))$ is the Laplace transform of some function.
	However, let us observe that $S_{\Phi,H}v \in L^\infty(\R^+;C(I))$ and $$\cL[S_{\Phi,H}v(x,\cdot)]=\frac{\Phi(\lambda)}{\lambda}L_Hv(x ,\Phi(\lambda)),$$ thus, being $\pdsup{}{x}{2}:C^2(I)\to C(I)$ a closed operator and $S_{\Phi,H}:L^\infty(\R^+;C^2(I))\to L^\infty(\R^+;C^2(I))$ well defined, we have, by \cite[Proposition $1.7.6$]{arendt}, that $S_{\Phi,H}v(\cdot,t)\in C^2(I)$ for any $t>0$ and
	\begin{equation*}
	\pdsup{}{x}{2}S_{\Phi,H}v(x,t)=\cL^{-1}_{\lambda \to t}\left[\pdsup{}{x}{2}\frac{\Phi(\lambda)}{\lambda}L_Hv(x ,\Phi(\lambda))\right](t)=\cF_{\Phi,H}v_\Phi(x,t),
	\end{equation*}
	concluding the proof.
\end{proof}
\begin{rmk}\label{rmk:multoFok}
	With the same argument as in the previous Lemma, we can prove that if $v \in L^\infty(\R^+;C(I))$ is a mild solution of Equation \eqref{nongenFP}, then $v(\cdot, t) \in C^2(I)$.
\end{rmk}
Now we are ready to prove the main result of this section.
\begin{thm}\label{thm:gain}
	Let $v \in L^\infty(\R^+;C(I))$ and $v_\Phi=S_\Phi v$, with $v$ satisfying the hypotheses of Proposition \ref{altformLh}, be a mild solution of Equation \eqref{genFP}. Suppose for fixed $t>0$ it holds $v(\cdot,t) \in C^2(I)$. Moreover, suppose that $$V'_{2,H}(\cdot)\pdsup{}{x}{2}v(x,\cdot)\in L^\infty(0,+\infty)$$ for any fixed $x \in I$. Then $v_\Phi$ is a classical solution of \eqref{genFP}.
\end{thm}
\begin{proof}
	By Proposition \ref{propalphato1} and Remark \ref{rmk:multoFok} we know that $v(\cdot, t) \in C^2(I)$.\\
	Now let us observe that, since we know that $v$ is mild solution of \eqref{nongenFP}, it holds
	\begin{equation*}
	\cL[v(x,\cdot)](\lambda)=\frac{1}{\lambda}v(x,0)+\frac{1}{2\lambda}\pdsup{}{x}{2}L_Hv(x,\lambda).
	\end{equation*}
	Since $v(\cdot,t) \in C^2(I)$ and $\pdsup{}{x}{2}:C^2(I)\to C(I)$ is a closed operator, we have, by \cite[Proposition $1.7.6$]{arendt},
	\begin{equation*}
	\pdsup{}{x}{2}L_Hv(x,\lambda)=L_H\left(\pdsup{}{x}{2}v(x,\cdot)\right)(\lambda).
	\end{equation*}
	Now, according to theorem's condition,  $V'_{2,H}(\cdot)\pdsup{}{x}{2}v(x,\cdot) \in L^\infty(0,+\infty)$, therefore we can define
	\begin{equation*}
	F(x,t)=\frac{1}{2}\int_0^tV'_{2,H}(s)\pdsup{}{x}{2}v(x,s)ds
	\end{equation*}
	and take the Laplace transform to obtain
	\begin{equation*}
	\cL[F(x,\cdot)](\lambda)=\frac{1}{2\lambda}L_H\left[\pdsup{}{x}{2}v(x,\cdot)\right](\lambda).
	\end{equation*}
	Hence we get
	\begin{equation*}
	\cL[v(x,\cdot)](\lambda)=\cL[v(x,0)+F(x,\cdot)](\lambda)
	\end{equation*}
	and then
	\begin{equation*}
	v(x,t)=v(x,0)+\frac{1}{2}\int_0^t V'_{2,H}(s)\pdsup{}{x}{2}v(x,s)ds.
	\end{equation*}
	In particular, $v(x,\cdot)$ is absolutely continuous and, taking almost everywhere the derivative in $t$, $v$ is a classical solution of \eqref{nongenFP}.\\
	Now let us consider the function $v_\Phi(x,t)-v(x,0)$ and, observing that $\bar{\nu}_\Phi$ is Laplace transformable with Laplace transform $\frac{\Phi(\lambda)}{\lambda}$, we have that
	\begin{equation*}
	\cL[\bar{\nu}_\Phi \ast (v_\Phi(x,\cdot)-v(x,0))]=\frac{\Phi(\lambda)}{\lambda}\cL[v_\Phi(x,\cdot)](\lambda)-\frac{\Phi(\lambda)}{\lambda^2}v(x,0).
	\end{equation*}
	Now, since $\partial_t v(x,t)=\frac{1}{2}\cF_Hv(x,t)$ and $\cF_Hv(x,\cdot)\in L^\infty(0,+\infty)$, also $\partial_t v(x,\cdot)\in L^\infty(0,+\infty)$ and we can apply $S_\Phi$ to it. By Proposition \ref{Proplaptrans} and \cite[Corollary $1.6.5$]{arendt} we have
	\begin{align*}
	\cL_{t \to \lambda}\left[\int_0^t S_\Phi\partial_t v(x,s)ds\right]&=\frac{\Phi^2(\lambda)}{\lambda^2}\cL[v(x,\cdot)](\Phi(\lambda))-\frac{\Phi(\lambda)}{\lambda^2}v(x,0)\\
	&=\frac{\Phi(\lambda)}{\lambda}\cL[v_\Phi(x,\cdot)](\lambda)-\frac{\Phi(\lambda)}{\lambda^2}v(x,0)\\
	&=\cL[\bar{\nu}_\Phi \ast (v_\Phi(x,\cdot)-v(x,0))],
	\end{align*}
	and then it holds
	\begin{equation*}
	\bar{\nu}_\Phi \ast (v_\Phi(x,\cdot)-v(x,0))(t)=\int_0^t S_\Phi\partial_t v(x,s)ds.
	\end{equation*}
	Thus we can differentiate on both sides to achieve, for almost any $t>0$,
	\begin{equation*}
	\partial_t^\Phi v_\Phi(x,t)=S_\Phi\partial_t v(x,t).
	\end{equation*}
	However, we also have, being $v_\Phi$ a mild solution of \eqref{genFP},
	\begin{align*}
	\cL\left[S_\Phi\partial_t v(x,\cdot)\right](\lambda)&=\Phi(\lambda)\cL[v_\Phi(x,\cdot)](\lambda)-\frac{\Phi(\lambda)}{\lambda}v(x,0)\\
	&=\frac{\Phi(\lambda)}{2\lambda}\pdsup{}{x}{2}\widehat{L}_H\cL[v_\Phi](x,\lambda).
	\end{align*}
	Hence, we have that $\frac{\Phi(\lambda)}{2\lambda}\pdsup{}{x}{2}\widehat{L}_H\cL[v_\Phi](x,\lambda)$ is the Laplace transform of something and then we can take the inverse Laplace transform to obtain
	\begin{equation*}
	S_\Phi\partial_t v_\Phi(x,t)=\frac{1}{2}\cF_{\Phi,H} v_\Phi(x,t).
	\end{equation*}
	Finally we get
	\begin{equation*}
	\partial_t^\Phi v_\Phi(x,t)=\frac{1}{2}\cF_{\Phi,H} v_\Phi(x,t),
	\end{equation*}
	concluding the proof.
\end{proof}
As a direct consequence, we can formulate  the following statement concerning the function $p_{\Phi,H}$.
\begin{cor}
	$p_{\Phi,H}$ is a classical solution of \eqref{genFP} for $I=\R^*$. Moreover $p_{\alpha,H}$ is a strong solution of \eqref{genFP}.
\end{cor}
\begin{rmk}
	Let us observe that if $v$ satisfies hypothesis $(a)$ of Proposition \ref{altformLh}, then the hypotheses of Theorem \ref{thm:gain} also imply that $v_\alpha=S_\alpha v$ is a strong solution of equation \eqref{genFP}.
\end{rmk}
\section{Uniqueness issues}\label{Sec5}
In this section we will discuss some uniqueness issues concerning mild and strong solutions. Let us stress, as we will observe later, that the uniqueness results concerning strong solutions can be adapted to the classical ones by extending the extremal point property in Proposition \ref{extremval} to less regular functions via a mollifying procedure.
\subsection{Isolation of mild solutions}
However, for mild solutions, we are not able to show uniqueness. We can still prove a form of \textit{isolation} result for mild solutions, i.e. the fact that mild solutions cannot be compared with respect to a suitable partial order.
\begin{defn}
	Let us define $\cS_\Phi$ the range of the operator $S_\Phi:L^\infty(\R^+;C(I)) \to L^\infty(\R^+;C(I))$ where $I=[a,b]$. Let $v_\Phi, w_\Phi \in L^\infty(\R^+;C(I))$ with $v_\Phi=S_\Phi v$ and $w_\Phi=S_\Phi w$. We say that $v_\Phi \preceq w_\Phi$ if and only if:
	\begin{itemize}
		\item $v \le w$ in $I \times \R^+$;
		\item There exist two constants $0<\varepsilon\le M$ such that for any $x \in I$ the function $(w-v)(x,\cdot)$ is increasing in $[0,\varepsilon]$ and decreasing in $[M,+\infty)$.
	\end{itemize}
	In particular $\preceq$ is a partial order on $\cS_\Phi$, that is well defined by injectivity of the operator $S_\Phi$.
\end{defn}
Now let us observe that we can recognize \eqref{ltgenFP} as a second order parametric ordinary differential equation. Therefore,  we can consider the Cauchy problem
\begin{equation}\label{CauprobgenFP}
\begin{cases}
\Phi(\lambda)\bar{v}_\Phi(x,\lambda)-\frac{\Phi(\lambda)}{\lambda}v_\Phi(x,0)=\frac{\Phi(\lambda)}{2\lambda}\pdsup{}{x}{2}\widehat{L}_{\Phi,H}\bar{v}_\Phi(x,\lambda) & (x,\lambda)\in I \times (0,+\infty)\\
v_\Phi(x,0)=f(x) & x \in I\\
\widehat{L}_{\Phi,H}\bar{v}_\Phi(a,\lambda)=g_1(\lambda) & \lambda>0 \\
\pd{}{x}\widehat{L}_{\Phi,H}\bar{v}_\Phi(a,\lambda)=g_2(\lambda) & \lambda>0,
\end{cases}
\end{equation}
which is the \textit{natural} Cauchy problem associated to mild solutions of equation \eqref{genFP}. What we want to show is that two different mild solutions of Equation \eqref{genFP} with the same \textit{boundary data} cannot be compared with $\preceq$. This is actually the aim of the following Theorem.
\begin{thm}
	Let $I=[a,b]$, $v,w \in L^\infty(\R^+;C(I))$ and consider $v_\Phi=S_\Phi v$ and $w_\Phi=S_\Phi w$ such that, denoting $\bar{v}_\Phi=\cL[v_\Phi]$ and $\bar{w}_\Phi=\cL[w_\Phi]$, these are solutions of the Cauchy problem \eqref{CauprobgenFP}. If $w_\Phi \preceq v_\Phi$, then $w_\Phi=v_\Phi$.
\end{thm}
\begin{proof}
	First of all, let us observe that, since all the operators involved are linear, $(v_\Phi-w_\Phi)$ is still a mild solution of \eqref{genFP}. Let us set $h_\Phi(x,t)=(v_\Phi(x,t)-w_\Phi(x,t))$. By the initial datum on $v$ and $w$, it holds $h_{\Phi}(x,0)=0$ and then
	\begin{equation*}
	 \Phi(\lambda)\bar{h}_\Phi(x,\lambda)=\frac{\Phi(\lambda)}{2\lambda}\pdsup{}{x}{2}\widehat{L}_{\Phi,H}\bar{h}_\Phi(x,\lambda),
	\end{equation*}
	where $\bar{h}_\Phi(x,\lambda):=\cL[h_\Phi(x,\cdot)](\lambda)$. Moreover, since $S_\Phi$ is linear, we can define $h(x,t)=v(x,t)-w(x,t)$ to obtain that $h \in \cS_\Phi$ with $h_\Phi=S_\Phi h$. Setting $\bar{h}(x,\lambda):=\cL[h(x,\cdot)](\lambda)$, we have
	\begin{equation}\label{secondordeq}
	2\Phi(\lambda)\bar{h}(x,\Phi(\lambda))=\pdsup{}{x}{2}L_{H}h(x,\Phi(\lambda)),
	\end{equation}
	that is a second order differential equation. Now we want to transform the previous second order differential equation in a system of first order ones and then write it in vector form. Let us define
	\begin{equation*}
	f(x,\lambda)=\pd{}{x}L_{H}h(x,\Phi(\lambda)) \qquad \mbox{ and } \qquad
	g(x,\lambda)=(L_Hh(x,\Phi(\lambda)),f(x,\lambda))
	\end{equation*}
	to rewrite \eqref{secondordeq} in the equivalent form
	\begin{equation*}
	\pd{}{x}g(x,\lambda)=(f(x,\lambda),2\Phi(\lambda)\bar{h}(x,\Phi(\lambda))).
	\end{equation*}
	We want to show that actually $g \equiv 0$ by using Gronwall inequality. To do this, let us observe that $f(a,\lambda)=0$ and $L_Hh(a,\Phi(\lambda))=0$, thus we have
	\begin{equation*}
	g(x,\lambda)=\int_a^x\pd{}{x}g(y,\lambda)dy
	\end{equation*}
	and then
	\begin{equation}\label{ineqpreGron}
	|g(x,\lambda)|\le \int_a^x\left|\pd{}{x}g(y,\lambda)\right|dy.
	\end{equation}
	Let us first estimate $L_Hh(x,\Phi(\lambda))$. We have
	\begin{align*}
	L_H h(x,\Phi(\lambda))&=\int_0^\varepsilon e^{-\Phi(\lambda)t}h(x,t)V'_{2,H}(t)dt\\
	&+\int_\varepsilon^M e^{-\Phi(\lambda)t}h(x,t)V'_{2,H}(t)dt\\
	&+\int_M^{+\infty} e^{-\Phi(\lambda)t}h(x,t)V'_{2,H}(t)dt\\
	&:=I_1+I_2+I_3.
	\end{align*}
	We want to achieve a lower bound for $L_H h(x,\Phi(\lambda))$, which is non-negative since $h$ is non-negative. Let us observe that $\min_{t \in [\varepsilon,M]}V'_{2,H}(t)=m>0$, thus there exists a constant $C_1>0$ such that
	\begin{equation*}
	I_2\ge C_1 \int_\varepsilon^Me^{-\Phi(\lambda)t}h(x,t)dt.
	\end{equation*}
	For $I_1$ and $I_3$ we need to use Chebyshev's integral inequality (see \cite{mitrinovic1993classical}). Concerning $I_1$, we get
	\begin{equation*}
	 I_1=\frac{1-e^{-\Phi(\lambda)\varepsilon}}{\Phi(\lambda)}\int_0^{\varepsilon}V'_{2,H}(t)h(x,t)d\left(\frac{1-e^{-\Phi(\lambda)t}}{1-e^{-\Phi(\lambda)\varepsilon}}\right)
	\end{equation*}
	where $d\left(\frac{1-e^{-\Phi(\lambda)t}}{1-e^{-\Phi(\lambda)\varepsilon}}\right)$ is a probability measure on $[0,\varepsilon]$. Thus we can use Chebyshev's integral inequality, since we can suppose $V'_{2,H}$ and $h(x,\cdot)$ to be comonotone in $[0,\varepsilon]$. Setting $C_2=\frac{\Phi(\lambda)}{1-e^{-\Phi(\lambda)\varepsilon}}\int_0^\varepsilon e^{-\Phi(\lambda)t}V'_{2,H}(t)dt>0$, we obtain
	\begin{equation*}
	I_1 \ge C_2 \int_0^{\varepsilon}e^{-\Phi(\lambda)t}h(x,t)dt.
	\end{equation*}
	Arguing in the same way for $I_3$, we have that there exists a constant $C_3>0$ such that
	\begin{equation*}
	I_3 \ge C_3 \int_M^{+\infty}e^{-\Phi(\lambda)t}h(x,t)dt.
	\end{equation*}
	Setting $C_4=\min_{i=1,2,3} C_i>0$, we obtain
	\begin{equation*}
	L_H h(x,\Phi(\lambda))\ge C_4 \bar{h}(x,\Phi(\lambda)).
	\end{equation*}
	Now let us define $k(x,\lambda)=\left|\pd{}{x}g(x,\lambda)\right|$ and observe that
	\begin{equation*}
	k(x,\lambda)=\sqrt{4\Phi^2(\lambda)\bar{h}^2(x,\Phi(\lambda))+f^2(x,\lambda)}.
	\end{equation*}
	Moreover, by using the previously obtained lower bound, setting $C_5=\min\left\{\frac{C_4^2}{4\Phi^2(\lambda)},1\right\}>0$, we have
	\begin{align*}
	|g(x,\lambda)|=\sqrt{(L_Hh(x,\Phi(\lambda)))^2+f^2(x,\lambda)}\ge C_5 k(x,\lambda).
	\end{align*}
	Plugging this inequality in Equation \eqref{ineqpreGron} and setting $C_6=C_5^{-1}$, we finally achieve
	\begin{equation*}
	k(x,\lambda)\le C_6 \int_a^x k(y,\lambda)dy.
	\end{equation*}
	Now we can use Gronwall's Inequality (see \cite{pachpatte1997inequalities}) to conclude that $k(x,\lambda)=0$. This implies that $\bar{h}(x,\Phi(\lambda))=0$. Now, considering $\lambda>0$, we have that $\Phi$ is invertible on the real line, thus we conclude that $\bar{h}(x,\lambda)=0$ for any $\lambda>0$. Finally, by injectivity of the Laplace transform, we obtain $h(x,t)=0$ for any $t>0$ and $x \in I$, that is what we wanted to prove.
\end{proof}
\subsection{Uniqueness of strong solutions}
In this subsection we want to address the problem of uniqueness of strong solutions. Concerning strong solutions, we can use the extremal point property given in Proposition \ref{extremval} to prove a weak maximum principle. To do this, we first need to show the following technical lemma.
\begin{lem}
	Let $v \in L^\infty(\R^+;C(I))$, $v_\Phi=S_\Phi v$ and $v_{\Phi,H}=S_{\Phi,H}v$. Then the following assertions are equivalent:
	\begin{itemize}
		\item[$i$] $(x_0,t_0) \in I \times \R^+$ is a maximum point of $v_\Phi$;
		\item[$ii$] $(x_0,t_0) \in I \times \R^+$ is a maximum point of $v_{\Phi,H}$.
	\end{itemize}
\end{lem}
\begin{proof}
	Let us show $ii \Rightarrow i$. Set $M=\sup_{t>0}V'_{2,H}(t)>0$ and suppose $(x_0,t_0)$ is a maximum point of $v_{\Phi,H}$. Then we have, for any $(x,t) \in I \times \R^+$,
	\begin{equation*}
	v_{\Phi,H}(x_0,t_0)-v_{\Phi,H}(x,t)=\int_0^{+\infty}V'_{2,H}(s)(v(x_0,s)f_\Phi(s,t_0)-v(x,s)f_\Phi(s,t))ds\ge 0.
	\end{equation*}
	On the other hand, it holds
	\begin{equation*}
	\int_0^{+\infty}V'_{2,H}(s)(v(x_0,s)f_\Phi(s,t_0)-v(x,s)f_\Phi(s,t))ds\le M (v_\Phi(x_0,t_0)-v_{\Phi,H}(x,t)),
	\end{equation*}
	thus we have
	\begin{equation*}
	M(v_\Phi(x_0,t_0)-v_{\Phi,H}(x,t))\ge 0,
	\end{equation*}
	concluding the proof.\\
	Now let us show $i \Rightarrow ii$. Fix $(x,t)\in I \times \R^+$. First of all let us suppose that there exists an increasing sequence $R_n \to +\infty$ and a decreasing sequence $\delta_n \to 0$ such that
	\begin{equation*}
	\int_{\delta_n}^{R_n}(v(x_0,s)f_\Phi(s,t_0)-v(x,s)f_\Phi(s,t))ds\ge 0.
	\end{equation*}
	Since $\min_{t \in [\delta_n,R_n]}V'_{2,H}(t)>0$, we achieve
	\begin{align*}
	v_{\Phi,H}(x_0,t_0)&-v_{\Phi,H}(x,t)\\
	&=\int_0^{\delta_n}V'_{2,H}(s)(v(x_0,s)f_\Phi(s,t_0)-v(x,s)f_\Phi(s,t))ds\\
	&+\int_{\delta_n}^{R_n}V'_{2,H}(s)(v(x_0,s)f_\Phi(s,t_0)-v(x,s)f_\Phi(s,t))ds\\
	&+\int_{R_n}^{+\infty}V'_{2,H}(s)(v(x_0,s)f_\Phi(s,t_0)-v(x,s)f_\Phi(s,t))ds\\
	&\ge \int_0^{\delta_n}V'_{2,H}(s)(v(x_0,s)f_\Phi(s,t_0)-v(x,s)f_\Phi(s,t))ds\\
	&+\int_{R_n}^{+\infty}V'_{2,H}(s)(v(x_0,s)f_\Phi(s,t_0)-v(x,s)f_\Phi(s,t))ds.
	\end{align*}
	Taking the limit as $n \to +\infty$ we obtain $v_{\Phi,H}(x_0,t_0)-v_{\Phi,H}(x,t) \ge 0$.\\
	Now let us suppose such sequences do not exist. Then, since $(x_0,t_0)$ is a maximum point of $v_{\Phi}$, this can happen only if
	\begin{equation*}
	\int_{0}^{+\infty}(v(x_0,s)f_\Phi(s,t_0)-v(x,s)f_\Phi(s,t))ds=0,
	\end{equation*}
	since $\int_{0}^{+\infty}(v(x_0,s)f_\Phi(s,t_0)-v(x,s)f_\Phi(s,t))ds>0$ goes in contradiction with the fact that the two aforementioned sequences do not exist.
	Moreover, there exist $\delta_0,R_0$ such that for any $\delta<\delta_0$ and $R>R_0$ it holds
	\begin{equation*}
	\int_{\delta}^R(v(x_0,s)f_\Phi(s,t_0)-v(x,s)f_\Phi(s,t))ds<0.
	\end{equation*}
	Since $\inf_{t \in (0,+\infty)}V'_{2,H}(t)=0$, we can consider $\delta_0$ so small and $R_0$ so big to obtain $\inf_{t \in (\delta_0,R_0)}V'_{2,H}(t)<1$. Consider any decreasing sequence $\delta_n \to 0$ such that $\delta_n<\delta_0$ and any increasing sequence $R_n \to +\infty$ such that $R_n>R_0$. Arguing as we did before, by using the fact that $\inf_{t \in (\delta_n,R_n)}V'_{2,H}(t)<1$, we achieve
	\begin{align*}
	v_{\Phi,H}(x_0,t_0)&-v_{\Phi,H}(x,t)\\
	&=\int_0^{\delta_n}V'_{2,H}(s)(v(x_0,s)f_\Phi(s,t_0)-v(x,s)f_\Phi(s,t))ds\\
	&+\int_{\delta_n}^{R_n}V'_{2,H}(s)(v(x_0,s)f_\Phi(s,t_0)-v(x,s)f_\Phi(s,t))ds\\
	&+\int_{R_n}^{+\infty}V'_{2,H}(s)(v(x_0,s)f_\Phi(s,t_0)-v(x,s)f_\Phi(s,t))ds\\
	&\ge \int_0^{\delta_n}V'_{2,H}(s)(v(x_0,s)f_\Phi(s,t_0)-v(x,s)f_\Phi(s,t))ds\\
	&+\int_{\delta_n}^{R_n}(v(x_0,s)f_\Phi(s,t_0)-v(x,s)f_\Phi(s,t))ds\\
	&+\int_{R_n}^{+\infty}V'_{2,H}(s)(v(x_0,s)f_\Phi(s,t_0)-v(x,s)f_\Phi(s,t))ds.
	\end{align*}
	Taking the limit as $n \to \infty$ we conclude the proof.
\end{proof}
Now we are ready to show the weak maximum principle for our equation.
\begin{thm}[\textbf{Weak maximum principle}]
	Let $\Phi$ be a Bernstein function that is regularly varying at $\infty$ of index $\alpha \in (0,1)$ and consider $v_\Phi=S_\Phi v$ a strong solution of \eqref{genFP} in $[a,b] \times \R^+$. Fix $T>0$ and define $\cO=[a,b] \times [0,T]$. Moreover, suppose that $S_\Phi\left(\frac{T-t}{T}\chi_{[0,T]}(t)\right)$ is $C^1((0,T])$ and $W^{1,1}(0,T)$. Let $\partial_p \cO$ be the parabolic boundary of $\cO$, i.e.
	\begin{equation*}
	\partial_p \cO=([a,b]\times \{0\}) \cup (\{a,b\}\times [0,T]).
	\end{equation*}
	Then it holds
	\begin{equation*}
	\max_{(x,t)\in \cO}v_\Phi(x,t)=\max_{(x,t)\in \partial_p\cO}v_\Phi(x,t)
	\end{equation*}
\end{thm}
\begin{proof}
	First of all, let us observe that for any constant $C \in \R$ it holds $v_\Phi+C=S_\Phi(v+C)$, $\cF_{\Phi,H}(v_\Phi+C)=\cF_{\Phi,H}v_\Phi$ and $\partial^\Phi (v_\Phi+C)=\partial^\Phi v_\Phi$. Thus if $v_\Phi$ is an inverse-subordinated strong solution of \eqref{genFP}, so it is also $v_\Phi+C$. In conclusion, we can suppose, without loss of generality, that $v_\Phi \ge 0$.\\
	Let us also recall that, by Lemma \ref{lem:weitoFok}, we have
	\begin{equation*}
	\cF_{\Phi,H}v_\Phi=\pdsup{}{x}{2}S_{\Phi,H}v.
	\end{equation*}
	Now let us suppose by contradiction $v_\Phi$ admits a maximum point $(x_0,t_0)$ belonging to $\mathring{\cO} \cup ((a,b)\times \{T\})$ and that $M=\max_{(x,t)\in \partial_p \cO}v_\Phi(x,t)<v_\Phi(x_0,t_0)$. Fix $\delta=v_\Phi(x_0,t_0)-M>0$ and define for any $(x,t)\in \cO$ the auxiliary function
	\begin{equation*}
	w_\Phi(x,t)=v_\Phi(x,t)+\frac{\delta}{2}S_\Phi \left(\frac{T-\tau}{T}\chi_{[0,T]}(\tau)\right)(t)
	\end{equation*}
	where $\chi_{[0,T]}(\tau)$ is the indicator function of the interval $[0,T]$. Since $\frac{T-t}{T} \in [0,1]$ as $t \in[0,T]$, it holds
	\begin{equation*}
	v_\Phi(x,t)\le w_\Phi(x,t)\le v_\Phi(x,t)+\frac{\delta}{2}
	\end{equation*}
	for any $(x,t)\in \cO$.\\
	Moreover, for any $(x,t) \in \partial_p \cO$ it holds
	\begin{equation*}
	w_\Phi(x_0,t_0)\ge v_\Phi(x_0,t_0)=\delta +M \ge \delta +v_\Phi(x,t)\ge \frac{\delta}{2}+w_\Phi(x,t)
	\end{equation*}
	and then, since $(x_0,t_0)\not \in \partial_p \cO$ and $w_\Phi$ is continuous in $[a,b]\times [0,T]$, $w_{\Phi}$ admits a maximum point $(x_1,t_1) \in \mathring{\cO} \cup ((a,b)\times \{T\})$.\\
	Now let us also recall that
	\begin{equation}\label{vPhiaswPhi}
	v_\Phi(x,t)=w_\Phi(x,t)-\frac{\delta}{2}S_\Phi\left(\frac{T-\tau}{T}\chi_{[0,T]}(\tau)\right)(t).
	\end{equation}
	Now we need to exploit $\partial_t^\Phi v_\Phi$ in terms of $\partial_t^\Phi w_\Phi$. To do this, we need to find the non-local derivative of $S_\Phi\left(\frac{T-\tau}{T}\chi_{[0,T]}(\tau)\right)(t)$.
	Set $g(t)=\frac{T-t}{T}\chi_{[0,T]}(t)$ and $g_\Phi(t)=S_\Phi g(t)$. First of all, since $w_\Phi=v_\Phi+g_\Phi$ and both $\partial^\Phi_t v_\Phi(x,t)$ and $\partial^\Phi g_\Phi(t)$ exist (where the latter exists since $g_\Phi(t)\in W^{1,1}(0,T)$), then also $\partial_t^\Phi w_\Phi$ is well defined.\\
	Now we want to determine $\partial^\Phi g_\Phi(t)$. Let us suppose a priori that $\partial^\Phi g_\Phi(t)$ is Laplace transformable. Then we have, since $g_\Phi(0)=1$,
	\begin{equation*}
	\cL[\partial^\Phi g_\Phi]=\Phi(\lambda)\cL[g_\Phi]-\frac{\Phi(\lambda)}{\lambda}=-\frac{\Phi(\lambda)(1-e^{-\Phi(\lambda)T})}{T\lambda \Phi(\lambda)}.
	\end{equation*}
	On the other hand, it holds
	\begin{equation*}
	\cL[S_\Phi \chi_{[0,T]}]=\frac{\Phi(\lambda)(1-e^{-\Phi(\lambda)T})}{\lambda \Phi(\lambda)},
	\end{equation*}
	thus we have
	\begin{equation*}
	\partial^\Phi g_\Phi(t)=-\frac{1}{T}\int_0^Tf_\Phi(s,t)ds.
	\end{equation*}
	Using last equality, together with identity \eqref{vPhiaswPhi}, we get
	\begin{equation}\label{dervPhiaswPhi}
	\partial^\Phi v_\Phi(x,t)=\partial^\Phi w_\Phi(x,t)+\frac{\delta}{2 T}\int_0^Tf_\Phi(s,t)ds.
	\end{equation}
	Now let us show that $w_\Phi$ belongs to the range of $S_\Phi$. Indeed, if we define $w(x,t)=v(x,t)+g(t)$, we obtain that $w_\Phi=S_\Phi w$.\\
	Now we want to express the action of the Fokker-Planck operator on $v_\Phi$ in terms of $w_\Phi$. To do this, since $g$ does not depend on $x$, just observe that
	\begin{equation*}\pdsup{}{x}{2}S_{\Phi,H}w(x,t)=\pdsup{}{x}{2}S_{\Phi,H}v(x,t).
	\end{equation*}
	Thus, thanks to Lemma \ref{lem:weitoFok}, we can rewrite equation \eqref{genFP} as
	\begin{equation*}
	\partial^\Phi w_\Phi(x,t)+\frac{\delta}{2T}\int_0^T f_\Phi(s,t)ds-\frac{1}{2}\pdsup{}{x}{2}S_{\Phi,H}w(x,t)=0.
	\end{equation*}
	Now let us observe that by hypotheses $w_\Phi$ belongs to $C^1$ in $(0,T]$ and $(x_1,t_1)$ is a maximum point for $w_\Phi$ belonging to $\mathring{\cO}\cup((a,b)\times \{T\})$, hence, by Proposition \ref{extremval}, we know that $\partial^\Phi w_\Phi(x_1,t_1)\ge 0$. Moreover, $(x_1,t_1)$ is also a maximum point for $S_{\Phi,H}w_(x,t)$, hence $\pdsup{}{x}{2}S_{\Phi,H}w(x_1,t_1)\le 0$. Thus we have
	\begin{equation*}
	\partial^\Phi w_\Phi(x_1,t_1)+\frac{\delta}{2T}\int_0^T f_\Phi(s,t_1)ds-\frac{1}{2}\pdsup{}{x}{2}S_{\Phi,H}w(x_1,t_1)\ge \frac{\delta}{2T}\int_0^T f_\Phi(s,t_1)ds>0,
	\end{equation*}
	which is a contradiction.
\end{proof}
\begin{rmk}
	Let us observe that, if $\Phi(\lambda)=\lambda^\alpha$, the fact that $\frac{T-t}{T}\chi_{[0,T]}(t)$ is Lipschitz implies that $S_\alpha\left(\frac{T-t}{T}\chi_{[0,T]}(t)\right)$ belongs to $C^1$ by Proposition \ref{derprop}. It holds $\left|\der{}{t}S_\alpha\left(\frac{T-t}{T}\chi_{[0,T]}(t)\right)\right| \sim C(\alpha,T)t^{-1+\alpha}$ as $t \to 0^+$, thus $S_\alpha\left(\frac{T-t}{T}\chi_{[0,T]}(t)\right)$ it belongs to $W^{1,1}(0,T)$.
\end{rmk}
As for classical parabolic equations, the weak maximum principle easily implies the following uniqueness result.
\begin{cor}
	Let $v_\Phi=S_\Phi v$ and $w_\Phi=S_\Phi w$ be strong solutions of Equation \eqref{genFP} in $\cO=(a,b)\times \R^+$. Suppose that for any $T>0$ the function $S_\Phi\left(\frac{T-t}{T}\chi_{[0,T]}(t)\right)$ belongs to $C^1((0,T])$. Then, if $v_\Phi(x,t)=w_\Phi(x,t)$ for any $(x,t)\in \partial_p\cO$, it holds $v_\Phi \equiv w_\Phi$.
\end{cor}
\begin{proof}
	Set $\cO_T=(a,b)\times (0,T)$ and $h=v-w$. Since all the involved operators are linear, the function $h_\Phi=v_\Phi-w_\Phi$ is still a strong solution of Equation \eqref{genFP}. Moreover $h_\Phi(x,t)=0$ for any $(x,t)\in \partial_p \cO_T$. Thus, by weak maximum principle, we have $h_\Phi\equiv 0$ on $\cO_T$. Since $T>0$ is arbitrary, we conclude the proof.
\end{proof}
In particular, this implies the following result.
\begin{cor}
	The probability density function $p_{H,\alpha}$ is the unique strong solution of Equation \eqref{genFP} in $(\R \setminus \{0\}) \times (0,+\infty)$ such that, for any fixed $t>0$, $\lim_{x \to \pm \infty}p_{H,\alpha}(x,t)=0$, $p_{H,\alpha}(0,t)=\int_0^{+\infty}p_H(0,s)f_\alpha(s,t)ds$ and $p_{H,\alpha}(x,0)=0$.
\end{cor}
\begin{rmk}
	Let us observe that all we stated in this Section for strong solutions can be extended to the case of classical solutions. Indeed, by using Friedrich's mollifiers, one can extend Proposition \ref{extremval} to the case in which $u \not \in C^1$. Thus, the weak maximum principle holds also for classical solutions without asking for $S_\Phi\left(\frac{T-t}{T}\chi_{[0,T]}(t)\right)$ to be in $C^1$.
\end{rmk}

\section*{Acknowledgements}
This research is partially supported by MIUR - PRIN 2017, project Stochastic Models for Complex Systems, no. 2017JFFHSH, by Gruppo Nazionale per il Calcolo Scientifico (GNCS-INdAM), by Gruppo Nazionale per l'Analisi Matematica, la Probabilit\`a e le loro Applicazioni (GNAMPA-INdAM). The research of the 2nd author was partially supported by ToppForsk project nr. 274410 of the Research Council of Norway with title STORM: Stochastics for Time-Space Risk Models.

\bibliographystyle{plain}
\bibliography{biblio}      
\end{document}